\documentclass[11pt]{article}

\usepackage{amsmath}
\usepackage{amsthm}
\usepackage{amsfonts}
\usepackage{bbm}

\usepackage[usenames]{color}

\newcommand{\mmp}{\mathbb{P}}

\newcommand{\od}{\overset{d}{=}}

\newcommand{\tp}{\overset{P}{\to}}

\newcommand{\me}{\mathbb{E}}
\newcommand{\mz}{\mathbb{Z}}
\newcommand{\mr}{\mathbb{R}}
\newcommand{\mn}{\mathbb{N}}

\newcommand{\lin}{\underset{n\to\infty}{\lim}}

\newcommand{\ve}{\varepsilon}

\newcommand{\mbR}{{\mathbb R}}
\newcommand{\mbN}{{\mathbb N}}

\newcommand{\sign}{\mathop{\mathrm{sign}}}
\newcommand{\wt}{\widetilde}

\newcommand{\cF}{{\cal F}}
\newcommand{\cE}{{\cal E}}

\newcommand{\be}{\begin{equation}}
\newcommand{\ee}{\end{equation}}

\DeclareMathOperator{\1}{\mathbbm{1}}

\newtheorem{thm}{Theorem}[section]
\newtheorem{lem}[thm]{Lemma}

\newtheorem{assertion}[thm]{Proposition}
\theoremstyle{definition}

\theoremstyle{remark}
\newtheorem{rem}[thm]{Remark}
\begin{document}
\title{A functional limit theorem for locally perturbed random walks}
%
\author{Alexander Iksanov\footnote{Faculty of Cybernetics, Taras Shevchenko National University of Kyiv, Ukraine\newline e-mail: iksan@univ.kiev.ua} \ \
and \ \ Andrey Pilipenko\footnote{Institute of Mathematics,
National Academy of Sciences of Ukraine, Kyiv, Ukraine\newline
e-mail: pilipenko.ay@yandex.ua} }
\maketitle
\begin{abstract}
\noindent A particle moves randomly over the integer points of the
real line. Jumps of the particle outside the membrane (a fixed
"locally perturbating set") are i.i.d., have zero mean and finite
variance, whereas jumps of the particle from the membrane have
other distributions with finite means which may be different for
different points of the membrane; furthermore, these jumps are
mutually independent and independent of the jumps outside the
membrane. Assuming that the particle cannot jump over the membrane
we prove that the weak scaling limit of the particle position is a
skew Brownian motion with parameter $\gamma\in [-1,1]$. The path
of a skew Brownian motion is obtained by taking each excursion of
a reflected Brownian motion, independently of the others, positive
with probability $2^{-1}(1+\gamma)$ and negative with probability
$2^{-1}(1-\gamma)$. To prove the weak convergence result we offer
a new approach which is based on the martingale characterization
of a skew Brownian motion. Among others, this enables us to
provide the explicit formula for the parameter $\gamma$. In the
previous articles the explicit formulae for the parameter have
only been obtained under the assumption that outside the membrane
the particle performs unit jumps.
\end{abstract}


\section{Introduction and main result}

Denote by $D:=D[0,\infty)$ the Skorokhod space of right-continuous
real-valued functions which are defined on $[0,\infty)$ and have
finite limits from the left at each point of the domain. We
stipulate hereafter that $\Rightarrow$ denotes weak convergence of
probability measures on $D$ endowed with the Skorokhod
$J_1$-topology.

For $x\in\mr$ and $(\xi_i)_{i\in\mn}$ a sequence of independent
identically distributed (i.i.d.) random variables which take
integer values and have zero mean and finite variance
$\sigma^2>0$, set $$S(0):=x,\quad S(n):=x+\xi_1+\ldots+\xi_n,\quad
n\in\mn.$$ Donsker's theorem states that
\begin{equation}\label{0}
U_n \ \Rightarrow \ W, \ \ n\to\infty,
\end{equation}
where $U_n(\cdot):=\sigma^{-1}n^{-1/2}S([n\cdot])$ and
$W:=(W(t))_{t\geq 0}$ is a Brownian motion. Like many other
authors (see references below and \cite{Paulin+Szasz:2010}) we are
interested in how the presence of a local perturbation of $(S(n))$
may influence \eqref{0}.

To define a local perturbation, we need more notation. Fix any
$m\in\mbN$ and set $A:=\{-m,-m+1,\dots,m\}$. For $j\in A$, denote
by $\eta_j$, $(\eta_{j,k})_{k\in\mn}$ i.i.d. integer-valued random
variables with a distribution that may depend on $j$. It is
assumed that the so defined random variables are independent of
$(\xi_i)$ and that $\eta_i$ and $\eta_j$ are independent whenever
$i\neq j$. For $x\in\mz$, define a random sequence
$(X(n))_{n\in\mn_0}$ by
$$
X(0)=x, \ \ X(n)=x+\sum_{k=1}^n \left(
\xi_k\1_{\{|X(k-1)|>m\}}+\sum_{|j|\leq
m}\eta_{j,k}\1_{\{X(k)=j\}}\right),\quad n\in\mn.
$$
Note that $(X(n))_{n\in\mn_0}$ is a homogeneous Markov chain with the transition probabilities
$$
p_{ij}:=
\begin{cases}
\mmp\{\xi=j-i\}, \ |i|> m;\\
\mmp\{\eta_i=j\}, \ |i|\leq m.
\end{cases}
$$
Assuming that the Markov chain $(X(n))_{n\in\mn_0}$ is irreducible\footnote{Here is a simple sufficient condition for irreducibility: $\mmp\{\xi_1=1\}>0$, $\mmp\{\xi_1=-1\}>0$, $\mmp\{\eta_j=1\}>0$ and $\mmp\{\eta_j=-1\}>0$ for all $j\in A$.}, set $\alpha_0:=0$,
$$
\alpha_{k}:=\inf\{i>\alpha_{k-1}: \ X(i)\in A \}, \ k\in\mn
$$
and $Y(k):=X(\alpha_k)$, $k\in\mn_0$. The sequence
$(Y(k))_{k\in\mn}$ is an irreducible homogeneous Markov chain.
Denote by $\pi:=(\pi_i)_{i\in A}$ its unique stationary
distribution. Note that $\pi_i>0$ for all $i\in A$. In the sequel
we shall use the standard notation $\me_\pi(\cdot):=\sum_{i\in
A}\pi_i\me (\cdot|Y(0)=i)$.

Recall that a {\it skew Brownian motion}
$W_\beta:=(W_\beta(t))_{t\geq 0}$ with parameter $\beta\in[-1,1]$
is a continuous Markov process with $W_\beta(0)=0$ and the
transition density
\begin{equation*}
p_t(x,y) = \varphi_t (x-y) + \beta\sign(y)\varphi_t (|x|+|y|),\
x,y \in \mathbb{R},
\end{equation*}
where $\varphi_t(x) = \frac{1}{\sqrt{2\pi t}} e^{-x^2/2t}$ is the
density of the normal distribution with zero mean and variance $t$
(see, for instance,  \cite{Lejay:2006}). The latter formula
enables us to conclude that $W_0$, $W_1$ and $W_{-1}$ have the
same distributions as $W$, $|W|$ and $-|W|$, respectively.

Our main result is given next.

\begin{thm}\label{thm_main}
In addition to all the aforementioned conditions assume that $\me
|\eta_j|<\infty$ for all $j\in A$ and that $|\xi_1|\leq 2m+1$ almost surely.
Then
$$
X_n\quad \Rightarrow\quad \sigma W_\gamma,\quad n\to\infty,
$$
where $X_n(t):=X([nt])/\sqrt{n}$ and $\gamma:= {\me_\pi
\big(X(1)-X(\alpha_1)\big)\over \me_\pi |X(1)-X(\alpha_1)|}$.
\end{thm}
\begin{rem}
Since $x$ in the definition of $(X(k))$ is arbitrary, the theorem remains valid if we replace the set $A$ with $A-j=\{-m-j,\ldots,m-j\}$ for any $j\in\mz$.
\end{rem}
\begin{rem}
Since $\me_\pi \big(X(1)-X(\alpha_1)\big)=\sum_{j\in
A}\pi_j\me\eta_j$, the condition $\me\eta_j=0$ for all $j\in A$ ensures that the limit process in Theorem \ref{thm_main} is a Brownian motion.
\end{rem}

Now we review briefly some related papers. The case $A=\{0\}$,
$1-\mmp\{\eta_0=-1\}=\mmp\{\eta_0=1\}=p\in [0,1]$,
$\mmp\{\xi_1=\pm 1\}=1/2$ has received considerable attention
\cite{Cherny+Shiryaev+Yor:2002, Harrison+Shepp:1981, Lejay:2006,
Szasz+Telcs:1981}. In \cite{Harrison+Shepp:1981} it is remarked
(without proof) that if $A$ and the distribution of $\xi_1$ are as
above, whereas $\eta_0$ has an arbitrary distribution which is
concentrated on integers and has a finite mean, then $\gamma=\me
\eta_0/ \me |\eta_0|$. To facilitate comparison of this equality
to the formula for $\gamma$ given in Theorem \ref{thm_main} we
note that in the present situation the stationary distribution
$\pi$ is degenerate at zero. The paper
\cite{Pilipenko+Prihodko:2014} is concerned with the case when
$A=\{0\}$, $\xi_1$ takes integer values (possibly more than two),
has zero mean and finite variance, whereas the distribution of
$\eta_0$ belongs to the domain of attraction of an $\alpha$-stable
distribution, $\alpha\in (0,1)$. The case when $m\in\mn$ is
arbitrary, $\mmp\{\xi_1=\pm 1\}=1/2$, and the variables $\eta_j$
are a.s.\ bounded, is investigated in \cite{Minlos+Zhizhina:1997,
Pilipenko+Prihodko:2011}. In \cite{Yarotsky:1999} the author
assumes that $\xi_1$ is a.s.\ bounded rather than having the
two-point distribution. The article \cite{Pilipenko+Prihodko:2015}
removes the assumption of a.s.\ boundedness of $\eta_j$, still
assuming that the distribution of $\xi_1$ is two-point.

The rest of the paper is structured as follows. In Section
\ref{decomp} we discuss our approach (which seems to be new in the
present context) which is based on the martingale characterization
of a skew Brownian motion. With this being an essential ingredient
the proof of Theorem \ref{thm_main} is finished in Section
\ref{l}. Some technical results are proved in Appendix.

\section{Proof of Theorem \ref{thm_main}}\label{main_proof}

\subsection{Decomposition of perturbed random walk}\label{decomp} We shall use the following martingale characterization of
a skew Brownian motion. Its proof can be found in
\cite{Kulik:2007}, see also \cite{Tsirelson}.
\begin{assertion}\label{thmKulik}
Suppose that a couple $(X,V):=(X_t, V_t)_{t\geq 0}$ of continuous
processes adapted to the filtration $(\cF_t)_{t\geq 0}$ satisfies
the following conditions:

\noindent 1) $V(0)=0$, $V$ is nondecreasing almost surely;

\noindent 2) processes $(M^\pm(t))_{t\geq 0}$ defined by
$$M^\pm(t):=X^\pm(t)-\frac{1\pm \beta}{2} V_t,\quad t\geq 0$$
are continuous martingales (with respect to $(\cF_t)$) with the
predictable quadratic variations
$$
\langle M^+\rangle_t=\int_0^t \1_{\{X_s\geq 0\}} {\rm
d}s\quad\text{и}\quad \langle M^-\rangle_t=\int_0^t \1_{\{X_s\leq
0\}} {\rm d}s,
$$
where $\beta\in[-1,1]$, $X^+_t=X_t\vee 0$ and $X^-_t=X^+_t-X_t$.

\noindent Then $X$ is a skew Brownian motion with parameter
$\beta$.
\end{assertion}

To prove Theorem \ref{thm_main} we decompose the perturbed random
walk $(X(n))$ into the sum of three summands. Roughly speaking,
these are given by the sums of jumps which are accumulated while
$(X(n))$ is staying in the sets $(m,\infty), (-\infty,-m)$ and
$[-m,m]$, respectively. It turns out that the first two summands
are martingales. Furthermore, their scaling limits are the
martingales $M^\pm$ appearing in Proposition \ref{thmKulik} (see
Lemma \ref{lem_conv} below). We analyze the third summand and its
scaling limit in Lemma \ref{lem_conv} and Section \ref{l}.

For convenience we assume that $X(0)=0$. The
general case can be treated similarly. For $n\in\mn_0$, set $\wt
X^\pm(n)=\pm X(n)\1_{\{\pm X(n)>m\}}$. Further, we put $\tau_0^\pm=0$,
$$\sigma_k^\pm=\inf\{i>\tau_k^\pm: \ \pm X(i)>m\},\
\tau_{k+1}^\pm:=\inf\{i>\sigma_k^\pm: \ \pm X(i)\leq m\},\quad
k\in\mn_0.$$

The subsequent presentation is essentially based on the following equality
\begin{eqnarray}\label{eq_predst}
\wt X^\pm (n)&=&\pm \sum_{k=1}^n \1_{\{\pm X(k-1)>m\}}\xi_k\pm
\sum_{i\geq 0}\left(X(\sigma_i^\pm)-X(\tau_i^\pm)
\right)\1_{\{\sigma_i^\pm\leq n\}}\mp \sum_{i\geq
0}X(\tau_i^\pm)\1_{\{\tau_i^\pm\leq n <
\sigma_i^\pm\}}\notag\\&=:& M^\pm(n)+L^\pm(n)\mp \sum_{i\geq 0}
X(\tau_i^\pm)\1_{\{\tau_i^\pm\leq n < \sigma_i^\pm\}}.
\end{eqnarray}
For $n\in\mn_0$, put
$$
M^\pm_n(t):={M^\pm([nt])\over \sqrt{n}},\quad
L^\pm_n(t):={L^\pm([nt])\over \sqrt{n}},\quad t\geq 0.
$$

The proofs of Lemma \ref{lem_comp} and Lemma \ref{lem_conv} given
below are relegated to Appendix.
\begin{lem}\label{lem_comp}
The sequence $(X_n^\pm, M_n^\pm,
L_n^\pm)_{n\in\mn}$ is weakly relatively compact on $D([0,T];
\mbR^6)$ for each $T>0$. Furthermore, each limit point
$(X_\infty^\pm, M_\infty^\pm,   L_\infty^\pm)$ of the sequence is
a continuous process satisfying  \be\label{eq_null} \int_0^T
\1_{\{X_\infty^\pm(t)= 0\}}{\rm d}t=0 \ \mbox{almost surely.} \ee
\end{lem}

\begin{lem}\label{lem_conv}
Let $(n_k)$ be a sequence such that
$$(X_{n_k}^\pm, M_{n_k}^\pm, L_{n_k}^\pm)\quad
\Rightarrow\quad (X_\infty^\pm, M_\infty^\pm, L_\infty^\pm),\quad
k\to\infty$$ on $D([0,T]; \mbR^6)$ for some $T>0$. Then

\noindent 1) the processes $L_\infty^\pm$ are nondecreasing almost surely and satisfy
\be\label{eq_XL} \int_0^T
\1_{\{X_\infty^\pm(t)> 0\}}{\rm d} L_\infty^\pm(t)=0\quad
\mbox{almost surely}; \ee

\noindent 2) the processes $(M_\infty^\pm(t))_{t\in[0,\,T]}$ are
continuous martingales with respect to the filtration
$(\cF_t)_{t\in [0,\,T]}$, where $\cF_t:=\sigma(X_\infty^\pm(s),
M_\infty^\pm(s), L_\infty^\pm(s), \ s\in [0,t])$, with the
predictable quadratic variations \be\label{eq_SX} \langle
M^\pm_\infty\rangle_t= \sigma^2
\int_0^t\1_{\{X^\pm_\infty(s)>0\}}{\rm d}s. \ee
\end{lem}

\subsection{Analysis of the processes $L_\infty^\pm$}\label{l}

If we can prove that \be L^+_\infty(t)={1+\gamma\over 1-\gamma}
L^-_\infty(t) \ \mbox{a.s.},\label{important} \ee then using
\eqref{eq_predst}, Lemma \ref{lem_comp} and the fact that the
absolute value of the last summand in \eqref{eq_predst} does not
exceed $m$ we conclude that
$$
X_\infty^\pm(t)=M_\infty^\pm(t)+\frac{1\pm
\gamma}{2}V_\infty(t),\quad t\geq 0\quad\text{a.s.},
$$
where
$$V_\infty(t):= {2\over 1+\gamma} L^+_\infty(t)={2\over 1-\gamma}
L^-_\infty(t).
$$
By Lemma \ref{lem_conv} and Proposition \ref{thm_main} $X$ is then
a skew Brownian motion with parameter $\gamma$.

Recalling the notation
$$
\alpha_0:=0,\quad \alpha_{k}:=\inf\{i>\alpha_{k-1}: \ X(i)\in
A\},\quad k\in\mn
$$
and $Y(n)=X(\alpha_n)$, $n\in\mn$, set $$\rho^\pm_k:=\pm
\big(Y(k+1)-Y(k)\big)\1_{\{\pm X(\alpha_k+1)\leq m\}}\pm
\big(X(\alpha_k+1)-Y(k)\big)\1_{\{\pm X(\alpha_k+1)>m\}},\quad
k\in\mn.$$
\begin{lem}\label{strong}
The following limit relation
$$\lin {\sum_{k=1}^n \rho^\pm_k\over
n}=\me_\pi\big(X(1)-X(\alpha_1)\big)^\pm$$ holds almost surely.
\end{lem}
The proof of the lemma is postponed until Appendix.

In view of
\begin{eqnarray*}
\bigg|L^\pm(n)-\sum_{k:\,\alpha_k\leq n}\rho^\pm_k\bigg|&=&
\bigg|\pm \sum_{i\geq
0}\big(X(\sigma_i^\pm)-X(\tau_i^\pm)\big)\1_{\{\sigma_i^\pm\leq
n\}}\\&\mp& \sum_{k:\,\alpha_k\leq n}
\bigg(\big(Y(k+1)-Y(k)\big)\1_{\{\pm X(\alpha_k+1)\leq m\}}\\&+&
\big(X(\alpha_k+1)-Y(k)\big)\1_{\{\pm X(\alpha_k+1)>m\}}
\bigg)\bigg|\\&\leq& 2m
\end{eqnarray*}
and Lemma \ref{rate}(a) we can invoke Lemma \ref{strong} to infer
$$
\lim_{n\to\infty}{L^+(n)\over L^-(n)} = {\me_\pi
(X(1)-X(\alpha_1))^+\over \me_\pi
(X(1)-X(\alpha_1))^-}={1+\gamma\over 1-\gamma}\quad \mbox{a.s.},
$$
thereby proving \eqref{important}. The proof of Theorem
\ref{thm_main} is complete.

\section{Appendix}

We start with an auxiliary result. For $n\in\mn_0$, denote by
$\nu(n)$ the sojourn time in $A$ of $(X(k))_{0\leq k\leq n}$,
i.e.,
$$\nu(n):=\sum_{k=0}^n\1_{\{|X(k)|\leq m\}}.$$
\begin{lem}\label{rate}

\noindent (a) $\lin \nu(n)=\infty$ almost surely;

\noindent (b)  $\me \nu(n)=O(\sqrt{n})$ as $n\to\infty$.
\end{lem}
\begin{proof}
Part (a) is obvious. Passing to the proof of part (b), for each
$j\in A=\{-m,\ldots, m\}$, we set
$$
\zeta^{(j)}_0:=\inf\{i\in\mn: X(i)=j \}$$ and
$$
\wt \zeta^{(j) }_k=\inf\{i> \zeta^{(j)}_k: |X(i)|>m\},\quad
\zeta^{(j) }_{k+1}=\inf\{i>\wt\zeta^{(j) }_k: X(i)=j\},\quad
k\in\mn
$$
with the standard convention that the infimum of the empty set
equals $+\infty$. Plainly, the so defined random variables are
stopping times w.r.t.\ the filtration generated by
$(X(k))_{k\in\mn_0}$. Furthermore, the random vectors $\{(\wt
\zeta^{(j) }_k- \zeta^{(j) }_k, \zeta^{(j) }_{k+1}- \wt \zeta^{(j)
}_k)\}_{k\in\mn}$ are i.i.d.

For typographical ease, we assume that $|X(0)|=|x|>m$ hereafter.
If the first entrance into $A$ following the $(l-1)$-st exit from
$A$, $l\in\mn$, occurs at the state $j_l$, then
$$\nu(n)\leq \sum_{l\geq 1} (\wt \zeta^{(j_l) }_{l-1}-
\zeta^{(j_l)}_{l-1})\1_{\{\zeta^{(j_l)}_{l-1}\leq
n\}}\quad\text{a.s.}$$ Hence $$\nu(n)\leq \sum_{|j|\leq m}
\sum_{k\geq 0}(\wt \zeta^{(j) }_k-
\zeta^{(j)}_k)\1_{\{\zeta^{(j)}_k\leq n\}}\leq \sum_{|j|\leq m}
\sum_{k\geq 0}(\wt \zeta^{(j) }_k- \zeta^{(j)}_k)
 \1_{\{(\zeta^{(j) }_{1}- \wt \zeta^{(j)}_0)+\ldots+ (\zeta^{(j) }_k- \wt \zeta^{(j)}_{k-1})\leq
 n\}}$$ and thereupon $$\me \nu(n)\leq \sum_{|j|\leq m}\me (\wt
\zeta^{(j) }_0- \zeta^{(j)}_0) \sum_{k\geq 0} \mmp\{(\zeta^{(j)
}_{1}- \wt \zeta^{(j)}_0)+\ldots+ (\zeta^{(j) }_k- \wt
\zeta^{(j)}_{k-1})\leq
 n\}$$ because, for $k\in\mn$, $\wt \zeta^{(j) }_k-
\zeta^{(j)}_k$ is independent of $\1_{\{(\zeta^{(j) }_{1}-
\wt\zeta^{(j)}_0)+\ldots+ (\zeta^{(j) }_k- \wt
\zeta^{(j)}_{k-1})\leq n\}}$ and has the same distribution as $\wt
\zeta^{(j) }_0- \zeta^{(j)}_0$. Thus, to complete the proof it
suffices to check that, for fixed $j\in A$,
\begin{equation}\label{1} \me (\wt \zeta^{(j) }_0- \zeta^{(j)}_0)<\infty
\end{equation}
and
\begin{equation} \label{2}
\underset{n\to\infty}{\lim\sup}\,n^{1/2}\sum_{k\geq 0}
\mmp\{(\zeta^{(j) }_{1}- \wt \zeta^{(j)}_0)+\ldots+ (\zeta^{(j)
}_k- \wt \zeta^{(j)}_{k-1})\leq n\}<\infty.
\end{equation}

\noindent {\sc Proof of \eqref{1}}. By using the mathematical
induction we can check that
$$\mmp\{\wt \zeta^{(j) }_0- \zeta^{(j)}_0>s\}\leq
\mmp\{|\eta_j+j|\leq m\}\big(\mmp\{\min_\ast|\eta_k+k|\leq
m\}\big)^{s-1},\quad s\in\mn,$$ where we write
$\underset{\ast}{\min}$ to mean that the minimum is taken over all
integer $k\in [-m,m]$ for which $\mmp\{|\eta_k+k|\leq m\}<1$. Such
indices $k$ do exist in view of the irreducibility. Thus, not only
does \eqref{1} hold, but also some exponential moments of $\wt
\zeta^{(j) }_0- \zeta^{(j)}_0$ are finite.

\noindent {\sc Proof of \eqref{2}}. Noting that
$$\{\pm X(\wt\zeta_0^{(j)})>m,\,\pm \xi_{\wt\zeta_0^{(j)}+1}\geq 0,\ldots, \pm \xi_{\wt\zeta_0^{(j)}+1}\pm\ldots\pm\xi_{\wt\zeta_0^{(j)}+n}\geq 0 \}
\subseteq \{\pm X(\wt\zeta_0^{(j)})>m,\,\zeta_1^{(j)}-\wt
\zeta^{(j)}_0>n\}$$ for $n\in\mn$ and setting
$p_j:=\mmp\{X(\wt\zeta_0^{(j)})>m\}$, we arrive at
\begin{eqnarray}\label{31}
\mmp\{\zeta_1^{(j)}-\wt \zeta^{(j)}_0>n\}&\geq& p_j
\mmp\{\xi_{\wt\zeta_0^{(j)}+1}\geq 0,\ldots,
\xi_{\wt\zeta_0^{(j)}+1}+\ldots+ \xi_{\wt\zeta_0^{(j)}+n}\geq
0\}\\&+&(1-p_j)\mmp\{\xi_{\wt\zeta_0^{(j)}+1}\leq 0,\ldots,
\xi_{\wt\zeta_0^{(j)}+1}+\ldots+\xi_{\wt\zeta_0^{(j)}+n}\leq
0\}\notag.
\end{eqnarray}
Observe that $(\xi_{\wt\zeta_0^{(j)}+1}+\ldots+
\xi_{\wt\zeta_0^{(j)}+k})_{k\in\mn}$ is a standard random walk.
Its jumps have zero mean and finite variance because these have
the same distribution as $\xi_1$. Hence
\begin{equation}\label{32}
\lin n^{1/2}\mmp\{\xi_{\wt\zeta_0^{(j)}+1}\geq 0,\ldots,
\xi_{\wt\zeta_0^{(j)}+1}+\ldots+ \xi_{\wt\zeta_0^{(j)}+n}\geq
0\}=c_+\in(0,\infty),
\end{equation}
\begin{equation}\label{33}
\lin n^{1/2}\mmp\{\xi_{\wt\zeta_0^{(j)}+1}\leq
0,\ldots,\xi_{\wt\zeta_0^{(j)}+1}+\ldots+
\xi_{\wt\zeta_0^{(j)}+n}\leq 0\}=c_-\in (0,\infty)
\end{equation}
(see, for instance, pp.\ 381-382 in \cite{BGT}). Using Erickson's
inequality (Lemma 1 in \cite{Erickson:1973}) we infer
$$\sum_{k\geq 0} \mmp\{(\zeta^{(j) }_{1}- \wt
\zeta^{(j)}_0)+\ldots+ (\zeta^{(j) }_k- \wt \zeta^{(j)}_{k-1})\leq
n\}\leq {2n\over \me \big((\zeta_1^{(j)}-\wt\zeta_0^{(j)})\wedge
n\big)}\leq {2\over \mmp\{\zeta_1^{(j)}-\wt\zeta_0^{(j)}>n\}}$$
which in combination with \eqref{31}, \eqref{32} and \eqref{33}
gives
$$\lim\sup_{n\to\infty}n^{1/2}\sum_{k\geq 0}
\mmp\{(\zeta^{(j) }_{1}- \wt \zeta^{(j)}_0)+\ldots+ (\zeta^{(j)
}_k- \wt \zeta^{(j)}_{k-1})\leq n\}\leq {2\over
p_jc_++(1-p_j)c_-}<\infty.$$ The proof of Lemma \ref{rate} is
complete.
\end{proof}

\begin{proof}[Proof of Lemma \ref{lem_comp}]
Weak relative compactness and continuity of the limit follow if we
can check that either of the sequences $(X_n^\pm)$, $(M_n^\pm)$
and $(L_n^\pm)$ is weakly relatively compact, and that their weak
limit points are continuous processes. Actually, verification for
$(L_n^\pm)$ is not needed, for (a) the absolute value of the last
summand in \eqref{eq_predst} does not exceed $m$; (b) $\sup_{t\geq
0}\,|X_n^\pm(t)-\wt X_n^\pm(t)|\leq m/\sqrt{n}$, where $$\wt
X^\pm_n(t):={\wt X^\pm([nt])\over \sqrt{n}},\quad t\geq 0.$$
Further, it is clear that instead of $(X_n^\pm)$ and $(M_n^\pm)$
we can work with $(X_n)$ and $(M_n)$, where, as usual,
$M_n:=M_n^+-M_n^-$.

According to Theorem 15.5 in \cite{Billingsley:1968} it suffices
to prove that
\begin{equation*}
\underset{\delta\to
0}{\lim}\,\underset{n\to\infty}{\lim\sup}\,\mmp\big\{\underset{|t-s|\leq
\delta,\, t,s\in [0,T]}{\sup}\,|X_n(t)-X_n(s)|>\varepsilon\big\}=0
\end{equation*}
and that
\begin{equation*}
\underset{\delta\to
0}{\lim}\,\underset{n\to\infty}{\lim\sup}\,\mmp\big\{\underset{|t-s|\leq
\delta,\, t,s\in [0,T]}{\sup}\,|M_n(t)-M_n(s)|>\varepsilon\big\}=0
\end{equation*}
for any $\varepsilon>0$ or, which is equivalent, that
\begin{equation}\label{eq1old}
\underset{\delta\to
0}{\lim}\,\underset{n\to\infty}{\lim\sup}\,\mmp\big\{\underset{|i-j|\leq
[\delta n],\, i,j\in
[0,[nT]]}{\sup}\,|X(i)-X(j)|>\varepsilon\sigma \sqrt{n}\big\}=0,
\end{equation}
\begin{equation}\label{eq1old1}
\underset{\delta\to
0}{\lim}\,\underset{n\to\infty}{\lim\sup}\,\mmp\big\{\underset{|i-j|\leq
[\delta n],\, i,j\in [0,[nT]]}{\sup}\,|M(i)-
M(j)|>\varepsilon\sigma \sqrt{n}\big\}=0.
\end{equation}
Furthermore, if $(X_n)$ and/or $(M_n)$ converge along a
subsequence, the corresponding limits have continuous versions.

Define a random sequence $(X^\ast(k))_{k\in\mn_0}$ by
\be\label{eq_X} X^\ast(k):=S(r(k)) +\sum_{|j|\leq
m}\sum_{i=1}^{r_j(k)}\eta_{j,\,i},\quad k\in\mn_0, \ee where
$$
r(0):=0,\quad r(n):=\sum_{i=0}^{n-1}\1_{\{|X^\ast(i)|>m\}}$$ and,
for each $j\in A=\{-m,\ldots,m\}$,
$$r_j(0):=0,\quad r_j(n)=\sum_{i=0}^{n-1}\1_{\{X^\ast(i)=j\}}.
$$
Then $(X^\ast(k))_{k\in\mn_0}$ is a Markov chain with
$X^\ast(0)=x$ and the same transition probabilities as the Markov
chain $(X(k))_{k\in\mn_0}$. Hence the distributions of the two
Markov chains are the same. This particularly implies that
\begin{equation}\label{eqd}
\sum_{|j|\leq m}r_j(n)\od
\nu(n-1)=\sum_{k=0}^{n-1}\1_{\{|X(k)|\leq m\}}
\end{equation}
for each $n\in\mn$, where $\od$ denotes equality of distributions.
Further, observe that
$$M(n)=\sum_{k=1}^n \xi_k\1_{\{|X(k-1)|>m\}}=\sum_{i=1}^n\big(X(i)-X(i-1)\big)\1_{\{|X(i-1)|>m\}},\quad
n\in\mn_0$$ and
$$M^\ast(n):=S(r(n))-x=\sum_{i=1}^n\big(X^\ast(i)-X^\ast(i-1)\big)\1_{\{|X^\ast(i-1)|>m\}},\quad
n\in\mn_0.$$ Since the sequences $(X(n))_{n\in\mn_0}$ and
$(X^\ast(n))_{n\in\mn_0}$ have the same distribution, so do
$(M(n))_{n\in\mn_0}$ and $(M^\ast(n))_{n\in\mn_0}$.

Relation \eqref{eq1old1} is a consequence of the following
\begin{eqnarray*}
&&\underset{\delta\to
0}{\lim}\,\underset{n\to\infty}{\lim\sup}\,\mmp\big\{\underset{|i-j|\leq
[\delta n],\, i,j\in
[0,[nT]]}{\sup}\,|M(i)-M(j)|>\varepsilon\sigma \sqrt{n}\big\}\\&=&
\underset{\delta\to
0}{\lim}\,\underset{n\to\infty}{\lim\sup}\,\mmp\big\{\underset{|i-j|\leq
[\delta n],\, i,j\in [0,[nT]]}{\sup}\,|M^\ast(i)-
M^\ast(j)|>\varepsilon\sigma \sqrt{n}\big\}\\&\leq&
\underset{\delta\to
0}{\lim}\,\underset{n\to\infty}{\lim\sup}\,\mmp\big\{\underset{|i-j|\leq
[\delta n],\, i,j\in [0,[nT]]}{\sup}\,|S(j)-S(i)
|>\varepsilon\sigma \sqrt{n}\big\}=0,
\end{eqnarray*}
where the last equality is implied by \eqref{0}.

Turning to the proof of \eqref{eq1old} we first show that, for any
$0\leq i,j \leq [nT]$,
\begin{equation}\label{22}
\underset{|i-j|\leq [\delta n]}{\sup}\,|X^\ast(i)-X^\ast(j)|\leq
2m+ 2\underset{|i-j|\leq [\delta
n]}{\sup}\,|S(i)-S(j)|+\max_{|l|\leq m}\max_{1\leq k\leq
r_l([nT])}|\eta_{l,\,k}|\quad \text{a.s.}
\end{equation}
By symmetry it is sufficient to investigate the case $0\leq
i<j\leq [nT]$.

\noindent If $|X^\ast(i)|\leq m$ and $|X^\ast(j)|\leq m$, then
$|X^\ast(i)-X^\ast(j)|\leq 2m$ a.s.

\noindent If $j-i \leq[\delta n]$ and $|X^\ast(k)|>m$ for all
$k\in\{i,\dots,j\}$, then $|X^\ast(i)-X^\ast(j)|\leq
\underset{|i'-j'|\leq [\delta n]}{\sup}\,|S(i')-S(j')|$.

\noindent Finally, assume that $j-i \leq [\delta n]$,
$X^\ast(i)>m$ and $X^\ast(j)<-m$ (the case $X^\ast(i)<-m$ and
$X^\ast(j)>m$ can be treated analogously). Set $\alpha:=\inf\{k>i:
X^\ast(k)\in [-m,m]\}$ and $\beta:=\sup\{k<j:
X^\ast(k)\in[-m,m]\}$. Then
\begin{eqnarray*}
|X^\ast(i)-X^\ast(j)|&\leq&  |X^\ast(i)-X^\ast(\alpha)| +
|X^\ast(\alpha)-X^\ast(\beta)|+
|X^\ast(\beta)-X^\ast(\beta+1)|\\&+&|X^\ast(\beta+1)-X^\ast(j)|\leq
2\underset{|i'-j'|\leq [\delta
n]}{\sup}\,|S(i')-S(j')|+2m\\&+&\max_{|l|\leq m}\max_{1\leq k\leq
r_l([nT])}|\eta_{l,\,k}|.
\end{eqnarray*}
Thus, \eqref{22} holds which entails
\begin{eqnarray*}\label{eq1old_2}
&&\mmp\big\{\underset{|i-j|\leq [\delta n],\, i,j\in
[0,[nT]]}{\sup}\,|X(i)- X(j)|>\varepsilon\sigma
\sqrt{n}\big\}\\&=& \mmp\big\{\underset{|i-j|\leq [\delta n],\,
i,j\in [0,[nT]]}{\sup}\,|  X^\ast(i)- X^\ast(j)|>\varepsilon\sigma
\sqrt{n}\big\} \\&\leq& \mmp\big\{2m+2\underset{|i-j|\leq [\delta
n],\, i,j\in [0,[nT]]}{\sup}\,|   S(i)-   S(j)| + \max_{|l|\leq
m}\max_{1\leq k\leq r_l([nT])}|\eta_{l,k}|>\varepsilon\sigma
\sqrt{n}\big\}.
\end{eqnarray*}

In view of \eqref{0} to complete the proof of \eqref{eq1old} it
remains to check that \be\label{eq_vysk} n^{-1/2}\max_{|l|\leq
m}\max_{1\leq k\leq r_l([nT])}|\eta_{l,\,k}|\tp 0,\quad
n\to\infty. \ee Using Boole's inequality (twice) and Markov's
inequality yields
\begin{eqnarray*}
\mmp\{ n^{-1/2}\max_{|l|\leq m} \max_{1\leq k\leq
r_l([nT])}|\eta_{l,\,k}|>\varepsilon\sqrt{n}\}&\leq&
\mmp\bigg\{\sum_{|j|\leq
m}r_j([nT])>x\sqrt{n}\bigg\}\\&+&\sum_{|l|\leq
m}\mmp\{\underset{1\leq k\leq
[x\sqrt{n}]+1}{\max}\,|\eta_{l,\,k}|>\varepsilon\sqrt{n}\}\\&\leq&
x^{-1}n^{-1/2}\me \sum_{|j|\leq
m}r_j([nT])\\&+&([x\sqrt{n}]+1)\sum_{|l|\leq
m}\mmp\{|\eta_{l,\,1}|>\varepsilon\sqrt{n}\}.
\end{eqnarray*}
Sending first $n\to\infty$ (taking into account \eqref{eqd}
together with Lemma \ref{rate} and the assumption $\lin
n\mmp\{|\eta_{l,1}|>n\}=0$) and then $x\to\infty$ we arrive at
\eqref{eq_vysk}.

It remains to prove \eqref{eq_null}. To this end, note that any
limit point $( X_\infty^\pm, M_\infty^\pm, L_\infty^\pm)$
satisfies
$$
X_\infty(t):=X^+_\infty(t)-X^-_\infty(t)=M^+_\infty(t)-M^-_\infty(t) + L^+_\infty(t)-L^-_\infty(t)=:M_\infty(t)+L_\infty(t).
$$
Representation \eqref{eq_X} together with Lemma \ref{rate} implies
that $M_\infty$ is a Brownian motion. Another appeal to
\eqref{eq_X} allows us to conclude that $L_\infty$ is a continuous
process of locally bounded variation. Hence \eqref{eq_null}
follows from the occupation time formula (Corollary 1.6 of Chapter
6 in \cite{Revuz+Yor:1999}) because $\langle X_\infty \rangle
(t)=\langle M_\infty\rangle(t)=t$ (see Proposition 1.18 of Chapter
4 in \cite{Revuz+Yor:1999}). The proof of Lemma \ref{lem_comp} is
complete.
\end{proof}

\begin{proof}[Proof of Lemma \ref{lem_conv}]
1) Since the prelimit processes $L_n^\pm$ are a.s.\ nondecreasing,
so are $L^\pm_\infty$.

For each $\ve>0$, denote by $f_\ve(x)$ a continuous nonnegative
functions that satisfies $f(x)=1$ for $x\geq \ve$, and
$f_\ve(x)=0$ for $x\leq \ve/2$. To prove \eqref{eq_XL} it is
sufficient to check that
$$
\int_0^T f_\ve(X^\pm_\infty(s)){\rm d}L^\pm_\infty(s)=0\quad
\mbox{a.s.}
$$
for each $\ve>0$ and then use $\lim_{\ve\to
0}f_\ve(x)=\1_{(0,\,\infty)}(x)$ together with Lebesgue's
dominated convergence theorem.

By Skorokhod's representation theorem there exist versions of the
original processes which converge a.s. Furthermore, the
convergence is locally uniform, for the limit processes are a.s.\
continuous. Hence we have (for versions)
$$
\int_0^T f_\ve({X^\pm_\infty(s)}){\rm
d}L^\pm_\infty(s)=\lim_{k\to\infty} \int_0^T
f_\ve({X^\pm_{n_k}(s)}){\rm d}L^\pm_{n_k}(s)=0\quad \mbox{a.s.}
$$
as desired.

\noindent 2) We only give the proof for $M^+_\infty$. We have to
check that (I) $(M^+_\infty(t))_{t\in [0,\,T]}$ is a martingale;
(II) $((M^+_\infty(t))^2-A(t))_{t\in [0,\,T]}$ is a martingale
where $A(t):=\sigma^2\int_0^t\1_{\{X^+_\infty(s)>0\}}{\rm d}s$,
$t\geq 0$. We concentrate on the proof of (II), for the proof of
(I) is similar but simpler.

Set $X_\infty:=X_\infty^+-X_\infty^-$. Observe that the
$\sigma$-algebra $\sigma(X_\infty(s), s\leq t)$ is generated by a
family of random variables
$$\{f(X_\infty(t_1),\cdots,X_\infty(t_j)) \ | \ j\in \mn,\ 0\leq
t_1<t_2<\dots<t_j\leq t, \ f\in C_b(\mr^j)\},$$ where $C_b(\mr^j)$
is the set of bounded continuous real-valued functions defined on
$\mr^j$. It thus suffices to verify
\begin{equation}\label{eq_mart}
\me
f\big(X_\infty(t_1),\dots,X_\infty(t_j)\big)\big((M_\infty^+(t))^2-A(t)-(M_\infty^+(t_j))^2+A(t_j)\big)=0
\end{equation}
for any $t\in [0,T]$, and $j\in\mn$, any $0\leq
t_1<t_2<\dots<t_j\leq t$ and any function $f\in C_b(\mr^j)$.

Put $\cF_0:=\{\oslash, \Omega\}$,
$\cF_k:=\sigma(X(i),\xi_i)_{1\leq i\leq k }$, $k\in\mn$ and
$${\cE}_k(n):=n^{-1}\bigg(\sum_{i=1}^k
\1_{\{X(i-1)>m\}}\xi_i\bigg)^2- \sigma^2 n^{-1} \sum_{i=1}^{k}
\1_{\{X (i-1)>m\}}.$$ Since $(\cE_k(n))_{k\in\mn_0}$ is a
martingale w.r.t. $(\cF_i)_{i\in\mn_0}$ we infer
\begin{eqnarray*}
\me \left(\big(M_n^+(t^{(n)})\big)^2-\sigma^2
\int_0^{t^{(n)}}\1_{\{X^+_n(s)>0\}}{\rm
d}s\bigg|\cF_{[nt_j]}\right)&=&\me \left({\cE}_{[nt]}
|\cF_{[nt_j]}\right)={\cE}_{[nt_j]}\\&=&(M_n^+(t_j^{(n)}))^2-\sigma^2\int_0^{t_j^{(n)}}\1_{\{X^+_n(s)>0\}}{\rm
d}s,
\end{eqnarray*}
where $t^{(n)}_k:=[nt_k]/n$, $k=1,\ldots, j$, $t^{(n)}:=[nt]/n$.
Hence
\begin{equation}\label{eq_mart1} \me
f\big(X_n(t^{(n)}_1),\ldots,
X_n(t^{(n)}_j)\big)\bigg(\big(M_n^+(t^{(n)})\big)^2-\big(M_n^+(t^{(n)}_j)\big)^2-\sigma^2
\int_{t^{(n)}_j}^{t^{(n)}}\1_{\{X^+_n(s)>0\}}{\rm d}s\bigg)=0.
\end{equation}
Fix $s\in [0,T]$. The sequence
$\bigg(\big(M_n^+(t^{(n)})\big)^2-\big(M_n^+(t^{(n)}_j)\big)^2-\sigma^2
\int_{t^{(n)}_j}^{t^{(n)}}\1_{\{X^+_n(s)>0\}}{\rm
d}s\bigg)_{n\in\mn}$ is uniformly integrable if we can show that
\begin{equation}\label{ineq}
\sup_{k\in\mn}\,\me {\cE}^2_{[ns]}(n)<\infty.
\end{equation}
The expression under the expectation sign in \eqref{eq_mart1},
with $n$ replaced by $n_k$, converges weakly, as $k\to\infty$, to
the expression under the expectation in \eqref{eq_mart}, whence
equality \eqref{eq_mart} follows by the aforementioned uniform
integrability.

While proving \eqref{ineq}, we assume, for simplicity, that $s=1$.
By the Marcinkiewicz-Zygmund inequality for martingales (Theorem 9
in \cite{Burkholder:1966})
\begin{equation}\label{burk}
\me {\cE}^2_n(n) \leq C\me \sum_{k=1}^n Z_k(n)^2
\end{equation}
for some constant $C>0$ which does not depend on $n$, where
$(Z_k(n))_{k\in\mn}$ are martingale differences defined by
$$Z_k(n):=n^{-1}\bigg((\xi_k^2-\sigma^2)\1_{\{X(k-1)>m\}}+2\xi_k\1_{\{X(k-1)>m\}}\sum_{i=1}^{k-1}\xi_i\1_{\{X(i-1)>m\}}\bigg),\quad
k\in\mn$$ (with the convention that $\sum_{i=1}^0\ldots=0$).
Setting $r:=\me (\xi_1^2-\sigma^2)^2<\infty$ we have
\begin{eqnarray*}
n^2\me Z_k(n)^2&\leq& 2\bigg(\me
(\xi_k^2-\sigma^2)\1_{\{X(k-1)>m\}}+4\me
\xi_k^2\1_{\{X(k-1)>m\}}\bigg(\sum_{i=1}^{k-1}\xi_i\1_{\{X(i-1)>m\}}\bigg)^2\bigg)\\&\leq&
2\bigg(r+4\sigma^2\me\bigg(\sum_{i=1}^{k-1}\xi_i\1_{\{X(i-1)>m\}}\bigg)^2\bigg)\leq
2(r+4\sigma^4(k-1))
\end{eqnarray*}
Using the last inequality and \eqref{burk} we arrive at
\eqref{ineq}. The proof of Lemma \ref{lem_conv} is complete.

\end{proof}

\begin{proof}[Proof of Lemma \ref{strong}]
Fix $x\in A$. It suffices to prove that the convergence holds
$\mmp_x$-a.s.\ rather than a.s. The subsequent proof is similar to
the proof of the strong law of large numbers for Markov chains
(see, for instance, p.~87 in \cite{Durrett:1999}). We only treat
$\rho_k:=\rho_k^+$.

Put $T_x^{(0)}:=0$ and, for $k\in\mn$, denote by $T_x^{(k)}$ the
time of the $k$th return of $(Y_j)$ to $x$. Also, for $k\in\mn$,
we set $\theta_k(x):=\sum_{j=T_x^{(k-1)}}^{T_x^{(k)}-1}\rho_j$ and
observe that the random variables $\theta_1(x)$,
$\theta_2(x),\ldots$ are independent and $\mmp_x$-identically
distributed. We have
\begin{eqnarray*}
\me_x \theta_1(x)&=&\sum_{y\in A}\sum_{j\geq 0}\me_x
\big(\big(Y(j+1)-Y(j)\big)\1_{\{X(\alpha_j+1)\leq m\}}\\&+&
\big(X(\alpha_j+1)-Y(j)\big)\1_{\{X(\alpha_j+1)>m\}}|Y(j)=y\big)\mmp\{Y(j)=y,
T_x^{(1)}>j\}\\&=& \sum_{y\in A}\me
\big(\big(Y(1)-Y(0)\big)\1_{\{X(1)\leq m\}}\\&+&
\big(X(1)-Y(0)\big)\1_{\{X(1)>m\}}|Y(0)=y\big)\sum_{j\geq
0}\mmp\{Y(j)=y, T_x^{(1)}>j\}\\&=&\me_xT_x^{(1)}\sum_{y\in A}\pi_y
\me \big(\big(Y(1)-Y(0)\big)\1_{\{X(1)\leq m\}}+
\big(X(1)-Y(0)\big)\1_{\{X(1)>m\}}|Y(0)=y\big)\\&=&\me_xT_x^{(1)}\me_\pi
\big(\big(X(\alpha_1)-Y(0)\big)\1_{\{X(1)\leq m\}}+
\big(X(1)-Y(0)\big)\1_{\{X(1)>m\}}\big)\\&=&\me_xT_x^{(1)}\me_\pi
\big(X(1)-X(\alpha_1) \big)\1_{\{X(1)>m\}}\\&=&
\me_xT_x^{(1)}\me_\pi \big(X(1)-X(\alpha_1) \big)^+
\end{eqnarray*}
having utilized Theorem 8.2 on p.~84 in \cite{Durrett:1999} for
the third equality, the last equality being a consequence of the
fact that on the event $\{X(1)<-m\}$ one has $X(1)<X(\alpha_1)$,
while on $\{X(1)\in [-m,m]\}$ one has $X(1)=X(\alpha_1)$. Using
the strong laws of large numbers for random walks and renewal
processes yields
$${\sum_{k=1}^{N_n(x)}\theta_k(x)\over n}= {N_n(x)\over
n}{\sum_{k=1}^{N_n(x)}\theta_k(x) \over N_n(x)}\quad \to\quad
\me_\pi \big(X(1)-X(\alpha_1)\big)^+,\quad n\to\infty$$
$\mmp_x$-a.s., where $N_n(x):=\#\{k\in\mn: T_x^{(k)}\leq n\}$. It
remains to note that
$$\bigg|\sum_{k=1}^n \rho_k-\sum_{k=1}^{N_n(x)}\theta_k(x)\bigg|\leq |\theta_{N_n(x)+1}(x)|\leq \max_{1\leq j\leq
n+1}\,|\theta_j(x)|,$$ and that, as $n\to\infty$, the right-hand
side divided by $n$ converges to zero $\mmp_x$-a.s.\ in view of
$\me |\theta_1(x)|<\infty$ and the Borel-Cantelli lemma.
\end{proof}

\end{document}